\documentclass[11pt]{amsart}

\usepackage{graphicx}
\usepackage{subfig}
\usepackage[percent]{overpic}

\usepackage{color}
\usepackage{latexsym,bm,xcolor}

\usepackage[latin9]{inputenc}
\usepackage{amsmath}
\usepackage{amsfonts}
\usepackage{amssymb}
\usepackage{amsthm}

\newcommand{\fakesection}[1]{%
	\par\refstepcounter{section}
	\sectionmark{#1}
	\addcontentsline{toc}{section}{\protect\numberline{\thesection}#1}
}

\newcommand{\R}{{\mathbb R}}
\newcommand{\C}{{\mathbb C}}

\newcommand{\Z}{{\mathbb Z}}

\renewcommand{\d}{\partial}

\newcommand{\half}{{\frac{1}{2}}}

\renewcommand{\phi}{\varphi}

\newcommand{\acal}{\mathcal{A}}
\newcommand{\bcal}{\mathcal{B}}

\newcommand{\la}{\lambda}

\newtheorem{theo}{{\sc Theorem}}
\newtheorem{cor}[theo]{{\sc Corollary}}
\newtheorem{lem}[theo]{{\sc Lemma}}
\newtheorem{prop}[theo]{{\sc Proposition}}
\newtheorem{conj}[theo]{{\sc Conjecture}}

\newenvironment{defn}{\medskip\noindent{\it Definition:\/} }{\medskip}

\newtheorem{rema}[theo]{Remark}
\title[Spectral rigidity of the ellipse]
{Eigenfunction asymptotics and  spectral Rigidity of the ellipse}

\author{Hamid Hezari }
\address{Department of Mathematics, UC Irvine, Irvine, CA 92697, USA} \email{hezari@math.uci.edu}
\author{Steve Zelditch}
\address{Department of Mathematics, Northwestern  University,
Evanston, IL 60208-2370, USA} \email{
zelditch@math.northwestern.edu}

\thanks{Research of the second author is partially supported by NSF grant DMS-1810747.}

\date{\today}

\begin{document}

\begin{abstract}  Microlocal defect measures for Cauchy data of Dirichlet, resp. Neumann, eigenfunctions of an ellipse $E$ are determined.  We prove that,  for any invariant curve for the billiard map  on  the boundary phase space $B^*E$  of an ellipse, there exists a sequence of
eigenfunctions whose Cauchy data concentrates on the invariant curve.   We use this result to give a new proof that ellipses are infinitesimally
spectrally rigid among $C^{\infty}$ domains with the symmetries of
the ellipse.
\end{abstract}

\maketitle

\fakesection{Introduction}

This note is part of a series  \cite{HeZ12, HeZ19} on the the inverse spectral
problem for elliptical domains $E \subset \R^2$.  In \cite{HeZ12},  it is shown, roughly speaking, that an isospectral deformation of
an ellipse through smooth domains (but not necessarily real analytic) which preserves the $\Z_2 \times \Z_2$ symmetry is trivial. In \cite{HeZ19} it is shown that ellipses
of small eccentricity are uniquely determined by their Dirichlet (or, Neumann) spectra among all $C^{\infty}$ domains, with no analyticity
or symmetry assumptions imposed.  In both \cite{HeZ12, HeZ19},
the main spectral tool is the wave trace singularity expansion and the special  form it takes in the case of ellipses. 
In this article, we take the dual approach of  studying the asymptotic concentration in the phase space $B^* \partial E$ of the Cauchy data
$ u^b_{j} $  of  Dirichlet (or, Neumann) eigenfunctions $u_j$ of elliptical domains  in the unit coball bundle of
the boundary $\partial E$.    In Theorem \ref{LOCAL}, 
 we show that,  for every regular  rotation number of the billiard map in the `twist interval', there
exists a sequence of eigenfunctions whose Cauchy data concentrates on the invariant curve with that rotation number  in $B^* \partial E$.  The proof
uses the classical separation of variables and one dimensional WKB analysis.

Before stating the results we introduce some notation and background. An orthonormal basis of  Dirichlet (resp. Neumann) eigenfunctions  in a bounded, smooth
Euclidean plane domain $\Omega$  is denoted  by 
$$\left\{ \begin{array}{l} (\Delta + \lambda_j^2) \phi_j = 0, \;\; \langle \phi_j, \phi_k \rangle : = \int_{\Omega} \phi_j \bar{\phi_k} dx, \\ \\ 
\phi_j |_{\partial \Omega} = 0, \; (\rm{resp.} \; \partial_{\nu} \phi_j |_{\partial \Omega} = 0), \end{array} \right. $$
where as usual $\partial_{\nu}$ denotes the inward unit normal. The semi-classcial Cauchy data is denoted by,
\begin{equation} \label{ujbdef} u_j^b : = \left\{ \begin{array}{ll} \phi_j |_{\partial \Omega},
& \mbox{Neumann} \\ & \\
\lambda_j^{-1} \partial_{\nu} \phi_j |_{\partial \Omega}, & \mbox{Dirichlet}.
\end{array} \right. \end{equation}
The Cauchy data are eigenfunctions of  the semi-classical eigenvalue problem, $N(\lambda_j) u_j^b = u_j^b$, where $N(\lambda)$ is a semi-classical Fourier integral operator quantizing the billiard map $\beta: B^*\partial \Omega \to B^*\partial \Omega$ (see \cite{HZ} for the precise statement).

We are interested here in the quantum limits of the Cauchy data \eqref{ujbdef} of an orthonormal basis of eigenfunctions of an ellipse, i.e. in the asymptotic limits
of the matrix elements
\begin{equation} \label{ME} \rho_j^b(Op_{\hbar}(a)):= \frac{\langle Op_{\hbar}(a)  u_j^b, u_j^b
\rangle}{\langle u_j^b, u_j^b \rangle},\;\; (\hbar  = \hbar_j = \lambda_j^{-1})  \end{equation}
of zeroth order semi-classical pseudo-differential operators $ Op_{\hbar}(a)  $ on $\partial E$ with respect to the $L^2$-normalized Cauchy data of eigenfunctions.
We note that $\rho_j^b$ is normalized so that $\rho_j^b(I) = 1$ and is a positive linear functional, hence all possible weak* limits are probability measures
on the unit coball bundle $B^*\partial \Omega$. Moreover, $\rho_j(N(\lambda)^* Op_{\hbar}(a) N(\lambda)) = \rho_j(Op_{\hbar}(a))$, so that the quantum limits
are quasi-invariant under the billiard map (see \cite{HZ} for precise statements). In Theorem  \ref{LOCAL} we determine the quantum limits of
sequences in \eqref{ME} for an ellipse. The proof uses many of the prior results on  WKB formulae for ellipse eigenfunctions, especially
those of \cite{KR60, WWD,Sie97}.


In large part, our interest in matrix elements \eqref{ME} owes to the fact that the  Hadamard variational formulae for eigenvalues of the Laplacian with Dirichlet boundary condition
expresses the eigenvalue variations as the special matrix elements \eqref{ME} given by,  \begin{equation} \label{MERHO}  \int_{\partial E} \dot{\rho} \;\; |u^b_{j}|^2 ds  \end{equation}  of the domain 
variation $ \dot{\rho}$ (not to be confused with $\rho_j$) against squares of the Cauchy data (see Section
\ref{HD}). 
As stated in Corollary \ref{LOCALCOR}, the limits of such integrals over all
 possible subsequences of eigenfunctions determines the  `Radon transform' of $\dot{\rho}$
over all possible invariant curves for the billiard map.  Under an infinitesimal isospectral deformation, all of the limits are zero. 
We use this result to give   a new proof of the spectral rigidity result in \cite{HeZ12}; see Theorem \ref{RIGID} and Corollary \ref{RIGIDCOR}.  

The principal  motivation for studying the inverse Laplace spectral problems for ellipses stems from the  Birkhoff conjecture 
that ellipses are the only bounded plane domains with completely integrable billiards. 
  Strong recent results, due to A. Avila, J. de Simoi,  V. Kaloshin, and A. Sorrentino
  \cite{AdSK,KS18}
have proved local versions of the Birkhoff conjecture using a weaker notion of integrability known as `rational integrability', i.e. 
that periodic orbits come in one-parameter families, namely invariant curves of the billiard map with rational rotation number. In this
article, Bohr-Sommerfeld invariant curves play the principal role rather than curves of periodic orbits.

\subsection{Statement of results}

The first result pertains to concentration of Cauchy data of sequences $\phi_{j}$ of Dirichlet (resp. Neumann) eigenfunctions on invariant curves
of the billiard map of an ellipse. We denote $E$ by
$\frac{x^2}{a^2}  +\frac{ y^2}{b^2} \leq 1$, $0 \leq b < a$, and choose the elliptical coordinates $(\rho, \vartheta)$
by
$$(x, y) = (c\cosh \rho \cos \vartheta, c \sinh \rho  \sin \vartheta). $$
Here,
$$c = \sqrt{a^2 - b^2}, \;\; 0 \leq \rho \leq \rho_{\max} = \cosh^{-1} (a/c), \; \;  0 \leq \vartheta
\leq 2 \pi. $$ We denote the angular Hamiltonian, which we will also call the action, by 
$$ I = {p_\vartheta^2}/c^2 + \cos^2 \vartheta.$$
The invariant curves of $\beta$ are the level sets of $I$. The range of $I$ is called the action interval. There is a natural measure $d\mu_\alpha$ on each level set $I=\alpha$ called the Leray measure which is invariant under $\beta$ and the flow of $I$. We refer to Section \ref{BACKGROUND} for detailed definitions and properties involving the billiard map of an ellipse, actions, invariant curves, and the Leray measure. 

 \begin{theo} \label{LOCAL} Let $E$ be an ellipse.  For any $\alpha$  in the action interval 
 of the billiard map of $E$, there exists a sequence of separable (in elliptical coordinates) eigenfunctions $\{\phi_{j} \}$ of eigenvalue $\lambda_j^2$  whose Cauchy data  
 concentrates on the level set $\{I = \alpha\}$, in the sense that, for any zeroth order semi-classical pseudo-differential operator 
 $Op_{\hbar}(a)$ on $B^* \partial E$ with principal symbol $a_0$,
 \begin{equation} \label{alphaformintro } \frac{\langle Op_{\hbar_j}(a)  u_j^b, u_j^b
\rangle}{\langle u_j^b, u_j^b \rangle} \to \frac{ \int_{I = \alpha} a_0 d\nu_{\alpha}} { \int_{I = \alpha} d\nu_{\alpha}}, \quad (h_j =\lambda_j^{-1} \to 0) \end{equation} 
where
 \begin{equation} \label{dnu} 
 d \nu_\alpha =\left\{ \begin{array}{ll}  \frac{1}{\sqrt{c^2(\cosh^2 \rho_{\max} - \cos^2 \vartheta)}} d \mu_\alpha, \quad & \mbox{Dirichlet}, \\&\\ \small \sqrt{c^2(\cosh^2 \rho_{\max} - \cos^2 \vartheta)}   \; d \mu_\alpha, \quad & \mbox{Neumann}. \end{array} \right.
\end{equation} 
\end{theo}

 In particular,
\begin{cor} \label{LOCALCOR} In the special case when the symbol $a(\vartheta, p_\vartheta) = \dot{\rho}(\vartheta)$ is only a function of the base variable $\vartheta$,
$$\frac{\int_{\partial E} \dot{\rho} \;\; |u^b_{j}|^2 ds}{\int_{\partial E} |u^b_{j}|^2 ds} \to \frac{\int_{I = \alpha} \dot{\rho} \, d\nu_{\alpha}}{ \int_{I = \alpha} d\nu_{\alpha}}\,
,$$
where $ds = \small \sqrt{c^2(\cosh^2 \rho_{\max} - \cos^2 \vartheta)} \, d \vartheta$ is the arclength measure. 
\end{cor}

\begin{rema}
	If we denote $\eta$ to be the symplectic dual variable of the arclength $s$, then our quantum limit can be expressed as 
	 \begin{equation*} \label{dnub} 
	d \nu_\alpha =\left\{ \begin{array}{ll} \sqrt{1-|\eta|^2} \, d \mu_\alpha, \quad & \mbox{Dirichlet}, \\&\\ \small  \frac{1}{\sqrt{1-|\eta|^2}}  \; d \mu_\alpha, \quad & \mbox{Neumann}. \end{array} \right.
	\end{equation*} 
	For the proof, see our computation of ${1-|\eta|^2}$ in the proof of Corollary \ref{ZERO}. 
	
	The appearance of the (non-invariant) factors  $\sqrt{1-|\eta|^2}$ and $1/ \sqrt{1-|\eta|^2}$ is consistent with the result of \cite{HZ}, where the quantum limits of boundary traces of ergodic billiard tables are studied. 
\end{rema}

 To our knowledge, Theorem \ref{LOCAL} is the first result on microlocal defect measures of Cauchy data of eigenfunctions
in non-ergodic cases. See Section \ref{RELATED} for related results. 
 One of the difficulties in determining the  limits of \eqref{ME} is that the Cauchy data $u_j^b$ are not $L^2$ normalized. 
It is shown in \cite[Theorem 1.1]{HT} that there exists $C, c > 0$ so that  $$ c \leq ||\lambda_j^{-1} \partial_{\nu} \phi_j||_{L^2(\partial \Omega)} \leq C $$for Dirichlet eigenfunctions of  Euclidean plane domains (and more general non-trapping cases). Hence the $L^2$ normalization in \eqref{ME} is rather mild. On the other hand, the corresponding inequalities do not hold in general
for Neumann eigenfunctions.  As pointed out in \cite[Example 7]{HT},  there are simple counter-examples to any constant upper bound on the unit disc (whispering
gallery modes). There do exist positive lower bounds for convex Euclidean domains. Hence, in  the case of an ellipse, the $L^2$ normalization in
\eqref{ME} is necessary to obtain limits.

\subsection{Spectral rigidity}

Before stating the results,  we review the main definitions.
An isospectral deformation of a plane domain $\Omega_0$  is a
one-parameter family $\Omega_t$ of plane domains for which the
spectrum of the Euclidean Dirichlet (or Neumann, or Robin) Laplacian
$\Delta_t$ is constant (including multiplicities). The deformation
is said to be a $C^1$ deformation through $C^{\infty}$ domains if
each $\Omega_t$ is a $C^{\infty}$ domain and the map $t \to
\Omega_t$ is $C^1$. We parameterize the boundary $\partial
\Omega_t$ as the image under the  map \begin{equation}
\label{rhodef}  x \in
\partial \Omega_0 \to x + \rho_t(x) \nu_x, \end{equation} where  $\rho_t \in
C^1([0, t_0], C^{\infty}(\partial \Omega))$. The first variation
is defined to be $\dot{\rho}(x) : = \frac{d}{dt}{|_{t=0}}
\rho_t(x)$.  An isospectral deformation is said to be trivial if
$\Omega_t = \Omega_0$ (up to isometry) for sufficiently small $t$.
A domain $\Omega_0$ is said to be spectrally rigid if all
$C^{\infty}$ isospectral deformations are trivial.

In \cite{HeZ12} the authors proved a somewhat weaker form of
spectral rigidity for ellipses, with `flatness' replacing `triviality'.
Its main result is the infinitesimal spectral rigidity of ellipses
among $C^{\infty}$  plane domains with the symmetries of an
ellipse. We orient the domains so that the symmetry axes are the
$x$-$y$ axes. The symmetry assumption is then that $ \rho_t$ is
invariant under $(x, y) \to (\pm x, \pm y)$. The variation is
called infinitesimally spectrally rigid if $\dot{\rho}_0 = 0$.

The  main result of \cite{HeZ12} is:

\begin{theo} \label{RIGID} Suppose that $\Omega_0$ is an ellipse, and that
$\Omega_t$ is a $C^1$ Dirichlet (or Neumann) isospectral
deformation of $\Omega_0$ through $C^{\infty}$ domains with  $\Z_2
\times \Z_2$ symmetry. Let $\rho_t$ be as in (\ref{rhodef}).  Then
$\dot{\rho} = 0$.

\end{theo}

\begin{cor} \label{RIGIDCOR} Suppose that $\Omega_0$ is an ellipse, and that
$t \to \Omega_t$ is a $C^1$   Dirichlet (or Neumann) isospectral
deformation through $\Z_2 \times \Z_2$ symmetric $C^{\infty}$
domains. Then $\rho_t$ must be flat at $t = 0$.
\end{cor}

The proof of Theorem \ref{RIGID} in \cite{HeZ12} used the variation of the wave trace. In the original posting  (arXiv:1007.1741) the authors used a
more classical Hadamard variational formula for variations of individual eigenvalues $\lambda_j(t)$, which appears in Section \ref{HD}.  The authors rejected this approach in favor of the one appearing in \cite{HeZ12} because it  was thought  that this argument was invalid when the eigenvalues were  multiple. When a multiple eigenvalue of a 1-parameter family $L_t$ of operators is perturbed, it splits into a collection of branches which  in general
are not differentiable in $t$. Moreover, the authors  assumed that the
variational formula would express the variation in terms of special separable eigenfunctions (see Section \ref{SEP}).  This created doubt that one could use the variational formula for individual eigenvalues. Instead, the authors used the variational formula for the trace of the wave group or equivalently for spectral projections, which are symmetric sums over all of the branches into which an eigenvalue splits. 

 However, as we show in this article,  the original variational formulae
were in fact correct even in the presence of multiplicities. The first point is that the non-differentiability issue does not arise for an isospectral deformation since no splitting occurs. Second, the vanishing of  the variation of eigenvalues  implies that the infinitesimal variation $\dot{\rho}$ is orthogonal to squares of all (Dirichlet) eigenfunctions in the eigenspace, and in particular the separable ones. More precisely, we prove that 
$$ \int_{\partial E} \dot{\rho} \;\; |u^b_{j}|^2 ds =0.$$
Then by Corollary \ref{RIGIDCOR}, we obtain that for every $\alpha$ in the action interval one has
\begin{equation} \label{AbelTransform} \int_{I=\alpha} \dot{\rho} \, d\nu_{\alpha}=0. \end{equation}
In the final step we calculate the measure $d\nu_{\alpha}$ and provide two proofs, one via inverting an Abel transform and another using the Stone-Weierstrass theorem, that \eqref{AbelTransform} implies $\dot{\rho} =0$. The proof in the Neumann case is similar and will be provided.

\subsection{\label{RELATED}  Related results and open problems }

Quantum limits of Cauchy data on manifolds with boundary have been studied in  \cite{HZ, CTZ} in the case where the billiard map $\beta$ is ergodic. To our knowledge, they have
not been studied before in non-ergodic cases.  Theorem \ref{LOCAL} shows that, as expected,  Cauchy data of eigenfunctions localize on invariant curves for the billiard map rather
than delocalize as in ergodic cases.

 $L^2$ norms of Cauchy data of eigenfunctions are studied in \cite{HT} in the Dirichlet case and in  \cite{BHT} in the  Neumann case. Further results on the quasi-orthonormality properties of Cauchy are studied in 
 \cite{BFSS, HHHZ}.

 The study of
 eigenfunctions in ellipses has a long literature and we make substantial use of it. In particular, we quote 
several articles in the physics literature, in particular
 \cite{WWD, Sie97}, and several in mathematics \cite{KR60,BB},   for detailed analyses  of eigenfunctions of the quantum ellipse.  There is also 
 a  series of articles of G. Popov and P. Topalov (see e.g. \cite{PT03,PT16}) on the use of KAM quasi-modes to study Laplace inverse
 spectral problems. In particular, in \cite{PT16}, Popov-Topalov also give a new proof of the rigidity result of \cite{HeZ12} and extend it to
 other settings. The approach in this article is  closely related to theirs, although it does not seem that the authors directly studied Cauchy
 data of eigenfunctions of an ellipse.

The multiplicity of Laplace eigenvalues of an ellipse appears to be largely an open problem. It is a non-trivial result of
C.L. Siegel that the multiplicities are either $1$ or $2$ in the case of  circular billiards; multiplicity $1$ occurs for, and only for,
rotationally invariant eigenfunctions.  The Laplacians of the  family of ellipses $\frac{x^2}{a^2} + \frac{y^2}{b^2} =1$ form an analytic
family containing the disk Laplacian, and one might try to use analytic perturbation theory to prove the following, 
\begin{conj} \label{MULT} For a generic class of ellipses the multiplicity of each eigenvalue is $\leq 2$. \end{conj}


 \subsection{Quantum Birkhoff conjecture}
 As mentioned above, ellipses have completely integrable billiards, and the 
classical Birkhoff conjecture is that elliptical billiards are the only completely integrable Euclidean billiards with convex bounded
 smooth domains. Despite much recent progress, the Birkhoff conjecture remains open.
 
 The eigenvalue problem on a Euclidean domain is often called `quantum billiards' in the physics literature (see e.g. \cite{WWD}).
 One could formulate  quantum analogues of the Birkhoff conjecture in several related but different ways. The   quantum analogue
of the Birkhoff conjecture is presumably that ellipses are the only `quantum integrable' billiard tables.  A standard notion of
quantum integrability is that the Laplacian
commutes with a second, independent,  (pseudo-differential) operator; we refer to \cite{TZ03} for background on quantum integrability. In Section \ref{SEP}, 
we explain  that the ellipse  is quantum integrable in that one may construct two commuting Schr\"odinger operators with the same eigenfunctions
and eigenvalues. The symbol
of the second operator then Poisson commutes with the symbol of the Laplacian, hence the billiard dynamics and billiard map are integrable. 
 A related version is that one can separate variables in solving the Laplace eigenvalue
problem.   It is not obvious that these two notions are equivalent;  in Section \ref{SEP} we use both  separation of variables and existence of  commuting
operators in studying the ellipse. Classical studies of separation of variables and its relation to integrability go back
 to C. Jacobi, P. St\"ackel,   L. Eisenhart and others,
and E.K. Sklyanin has studied the problem more recently. We do not make use of their results here.

Quantum integrability is much stronger than classical integrability, and one might guess that it is   simpler to prove the quantum
Birkhoff conjecture than the classical one.  Wave trace  techniques as  in \cite{HeZ12,HeZ19}  reduce Laplace  spectral determination and rigidity problems to dynamical inverse or rigidity results. The wave trace only `sees' periodic orbits and is therefore
well-adapted to results on rational integrability.  The dual approach through eigenfunctions
 studied in this article gives a different path to the quantum Birkhoff conjecture, in which rational integrability and periodic orbits play no role.

\subsection*{Acknowledgment} We thank Luc Hillairet for a discussion which prompted the revival of this note.

\section{\label{BACKGROUND} Classical billiard dynamics}

In this, and the next, section, we review some background definitions and  results on the
classical and quantum elliptical billiard. We follow the notation
of \cite{Sie97}; see also \cite{BB,WWD}.

An ellipse $E$ is a plane domain defined by,
$$\frac{x^2}{a^2}  +\frac{ y^2}{b^2} \leq 1, \;\;\; 0 \leq b < a. $$
Here, $a,$ resp. $b$, is the length of the semi-major (resp.
semi-minor) axis. The ellipse has foci at $(\pm c, 0)$ with $c =
\sqrt{a^2 - b^2}$ and its eccentricity is $e = \frac{c}{a}$. Its
area is $\pi a b$, which is fixed under an isospectral
deformation. We define elliptical coordinates $(\rho, \vartheta)$
by
$$(x, y) = (c\cosh \rho \cos \vartheta, c \sinh \rho  \sin \vartheta). $$
Here,
$$0 \leq \rho \leq \rho_{\max} = \cosh^{-1} (a/c), \; \;  0 \leq \vartheta
\leq 2 \pi. $$ The coordinates are orthogonal. The  lines $\rho =
{constant}$ are confocal ellipses and the lines $\vartheta = {constant}$ are
confocal hyperbolas. In the special case of the disc, we have $c =
0$, but we assume henceforth that $c \not=0$.

\subsection{ Action variables for the billiard flow}
The billiard flow on the ellipse $E$ is the (broken) geodesic flow of the Hamiltonian $H= {p_x^2 + p_y^2}$ on $T^*E$, which follows straight lines inside $E$ and reflects on $\partial E$ according to equal angle law of reflection. 

Action-angle variables on $T^* E$  are symplectic coordinates in which the billiard flow of the ellipse is given by Kronecker flows on the invariant Lagrangian 
submanifolds. We refer to \cite{Ar} for the general principles and to \cite{Sie97} for the special case of the ellipse. 
Let $p_\rho$ and $p_\vartheta$ be the symplectic dual variables  corresponding to the elliptic coordinates $\rho$ and $\vartheta$, respectively.  The two conserved quantities of the system are the energy (the Hamiltonian) $H$ and the angular Hamiltonian $I$ (which we also call the action), given in the coordinates $(\rho, p_\rho, \vartheta, p_\vartheta)$, by
$$ H= \frac{p_\rho^2 + p_\vartheta^2}{c^2 (\cosh^2 \rho - \cos^2 \vartheta)} \quad \text{and} \quad  I = \frac{p_\vartheta^2 \cosh^2 \rho + p_\rho^2 \cos^2 \vartheta}{p_\rho^2 + p_\vartheta^2}.$$
 In the notation of \cite{Tab}, $$I = \cos^2 \theta \cosh^2 \rho
 + \sin^2 \theta \cos^2 \vartheta,$$
 where $\theta$ is the
angle between a trajectory of the billiard flow and a tangent
vector to the confocal ellipse with parameter $\rho$. Note also that by the notation of \cite{Sie97}, $I = 1+\frac{L_1L_2}{c^2H}$ where $L_1L_2$ is the product of
two angular momenta about the two foci. The values of $I$  are restricted to $$ 0 \leq I \leq \frac{a^2}{c^2} = \cosh^2 ( \rho_\text{max}).$$ The upper limit $I=\cosh^2 ( \rho_\text{max})$ corresponds to the motion along the boundary and the lower limit $I=0$ corresponds to the motion along the minor axis. Moreover, there are two different kinds of motion in the ellipse depending on the sign of $I$. For $1 < I  < \cosh^2 ( \rho_\text{max}) $ the trajectories have a caustic in the form of a confocal ellipse. For $ 0 < I <1$ the caustic of the motion is a confocal hyperbola and the trajectories cross the $x$-axis between the two focal points. Both kinds of motions are separated by a separatrix which consists of orbits with $I = 1$ that go through the focal points of the ellipse.

In terms of $H$ and $I$, the canonical momenta, are given by 
\begin{equation}\label{Momenta} p_\rho^2 = c^2  (\cosh^2 \rho - I) H  \quad \text{and} \quad  p_\vartheta^2 = c^2 ( I - \cos^2 \vartheta ) H. \end{equation}
Therefore, the action variables are 
\begin{align} \label{RadialAction}
 I_\rho & = \frac{1}{2\pi} \int p_\rho \, d \rho = \frac{c \sqrt{H}}{\pi} \int_{\cosh^2 \rho \geq I, 
 	  \rho \geq 0}  \sqrt{\cosh^2 \rho - I} \, d\rho, \\ 
\label{AngularAction} I_\vartheta & = \frac{1}{2\pi} \int p_\vartheta \, d \vartheta = \frac{c \sqrt{H}}{\pi} \int_{\cos^2 \vartheta \leq  I,  0 \leq \phi \leq \pi }  \sqrt{ I - \cos^2 \vartheta } \, d \vartheta. \end{align}
In fact these are the actions for the half-ellipse $ 0 \leq \phi \leq \pi$. The integrals can be calculated in terms of $I$ using elliptic integrals of first and second kind (See \cite{Sie97}). The actions will play a key role in Section \ref{BSQC} in the description of  Bohr-Sommerfeld quantization conditions for the eigenvalues of the Laplacian. 

	\begin{figure}\centering
			\subfloat {\includegraphics[width=.4\linewidth]{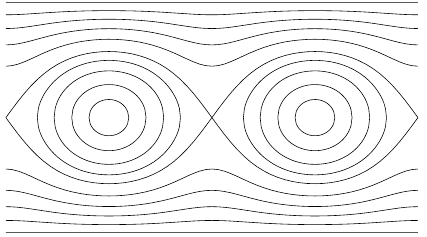}}\par 
		\subfloat{\includegraphics[width=.2\linewidth]{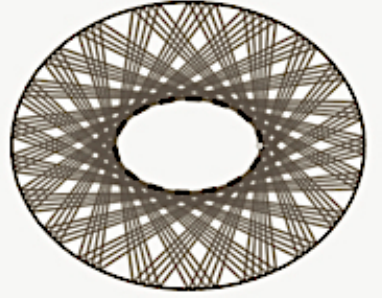}}
		\subfloat {\includegraphics[width=.2\linewidth]{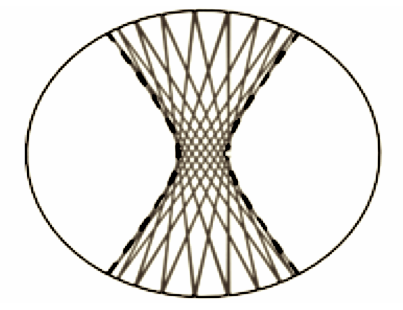}}
	    \caption{Invariant curves and caustics.}
		\label{Invariant curves}
	\end{figure}

\subsection{Billiard map, invariant curves, Leray measure, and action-angle variables}
The billiard map of an ellipse $E$ (or in general any smooth domain) is a cross section to the the billiard flow on $S^*_{\partial E} E$, which we always identify with $B^* \partial E$ and call it the phase space of the boundary. To be precise,
the billiard map $\beta$ is defined on $B^* \partial E$ as follows: given $(s, \eta) \in T^*\partial E$, with $s$ being the arc-length variable measured in the counter-clockwise direction from a fixed point say $s_0$, and $|\eta| \leq 1$, we let $(s, \zeta) \in S^* E$ be the unique inward-pointing unit covector at $s$ which projects to  $(s, \eta)$ under the map $T^*_{\partial E} E \to T^* \partial E$. Then we follow the geodesic (straight line) determined by $(s, \zeta)$ to the first place it intersects the boundary again; let $s' \in \partial E$ denote this first intersection. (If $|\eta| = 1$, then we let $s' = s$.)
 Denoting the inward unit normal vector at $s'$ by $\nu_{s'}$, we let $\zeta' = \zeta + 2 (\zeta \cdot \nu_{s'}) \nu_{s'}$ be the direction of the geodesic after elastic reflection at $s'$, and let $\eta'$ be the projection of $\zeta'$ to $T^*_{s'}Y$. Then we define
 $$
 \beta(s, \eta) = (s', \eta').
 $$

 A theorem of Birkhoff asserts that billiard map preserves the natural symplectic form $ds \wedge d \eta$ on $B^* \partial E$, i.e.
 $$\beta^* ( ds \wedge d \eta) = ds \wedge d \eta.$$  In the literature, the coordinates $(s, \theta)$ are commonly used for phase space of the boundary, where $\theta \in [0, \pi]$ is the angle that $\zeta$ makes with the positive tangent direction at $s$. In these coordinates,
 $$ds \wedge d \eta = \sin \theta \, d \theta \wedge ds $$
 An invariant set in $B^* \partial E$ is a set $C$ such that $\beta(C) = C$. 
 An invariant curve is a curve (connected or not) on the phase space that is invariant. The phase space $B^* \partial E$ of the ellipse $E$ is in fact foliated  with invariant curves. More precisely,
 \begin{lem} \label{ACTION} The invariant curves of the billiard map $\beta: B^*\partial E \to B^*\partial E$ are level sets 
 	of  $I: B^*\partial E \to \R$ defined by,
 	$$I =  \frac{p_{\vartheta}^2}{c^2} + \cos^2
 	\vartheta$$
 \end{lem}
\begin{proof} It follows quickly form the second equation of \eqref{Momenta} and that $H=1$ on $S^*\partial E$. \end{proof}
  Although $I_\vartheta$ is the classical angular action on $B^* \partial E$, but we shall call $I$ the action as it is more convenient and is related to $I_\vartheta$ via the one-to-one correspondence \eqref{AngularAction}.  As is evident from the Figure \ref{Invariant curves}, the separatrix curve $I=1$ divides the phase space into two types of open sets,  the exterior corresponding to trajectories with confocal elliptical caustics ($1 < I  < \cosh^2\rho_\text{max}$)  and the interior to trajectories with confocal hyperbolic caustics ($0<I <1$). 
  
  \subsubsection{Leray measure}
 On each level set $I=\alpha$ of $I$, there is a natural measure $d \mu_{\alpha}$ called the Leray measure which in invariant under $\beta$ and the flow generated by $I$.  In the symplectic 
coordinates $(\vartheta, p_{\vartheta})$, and on $I=\alpha$, it is given by
 $$d\mu_{\alpha} = \frac{d
	\vartheta \wedge d p_{\vartheta}}{d I}.$$
 Since $d \vartheta \wedge d I = \frac{\partial I}{\partial p_{\vartheta}} d \vartheta
\wedge d p_{\vartheta}$, we obtain that
\begin{equation} \label{Leray} d\mu_{\alpha} =  \frac{c^2}{2 p_{\vartheta}} \bigg\rvert_{I=\alpha}d \vartheta = \frac{c}{2} (\alpha- \cos^2 \vartheta)_+^{-1/2} d \vartheta. \end{equation}
Here, $x_+=x$ if $x >0$ and is zero otherwise. Up to a scalar multiplication, $d\mu_{\alpha}$ is a unique  measure that is invariant under $\beta$ and the flow of $I$.



\subsubsection{\label{alphaSECT} Action-angle variables and rotation number}

The billiard map has a Birkhoff normal form around each invariant
curve in $B^* \partial E$. That is, in the symplectically dual angle
variable $\iota$ to $I$, the billiard map has the form, $\beta(I,
\iota) = (I, \iota+ r(I))$, where $r$ is often called
the rotation number of the invariant curve. An explicit formula is
given for it in \cite{Tab} (3.5), \cite{CR} (section 4.3 (11)) and
\cite{K}. Then, if $0 < I < 1$,
\begin{equation*} \label{alphaforma} r(I) = \frac{\pi}{2 F(\sqrt{I})} F \left( \arcsin \left(
\frac{2 \tanh (\rho_\text{max}) \sqrt{\cosh^2\rho_\text{max} - I}}{\cosh^2\rho_\text{max} - I +
I \tanh^2\rho_\text{max}} \right), \sqrt{I} \right), \end{equation*} where
$$F(z, k)
 = \int_0^z \frac{d \tau}{\sqrt{1 - k^2 \sin^2 \tau}}, \;\; F(k) = F \left(\frac{\pi}{2}, k \right). $$
Also, if $1 < I <  \cosh^2(\rho_\text{max})$ then
\begin{equation*} \label{alphaform2} r(I) = \frac{\pi}{2 F(1 / \sqrt{I})} F \left(\arcsin \left(
\sqrt{I} \frac{2 \tanh(\rho_\text{max}) \sqrt{\cosh^2\rho_\text{max} - I}}{\cosh^2\rho_\text{max} - I + I \tanh^2\rho_\text{max}} \right), \frac{1}{\sqrt{I}} \right).
\end{equation*}

\begin{defn} \label{ACTROTDEF} We define the range of the action variable $I$ as the {\em action interval}, i.e. the interval $[0,  \cosh^2(\rho_\text{max})]$, and the range of $r(I)$ as the {\em rotation interval}. \end{defn}

\section{\label{SEP}Quantum elliptical billiard}

The Helmholtz equation in elliptical coordinates takes the form,
\begin{equation} \label{EQUIV} -\left (\frac{\partial^2}{\partial \rho^2} +
\frac{\partial^2}{\partial \vartheta^2} \right) \phi =  \lambda^2 c^2
(\cosh^2 \rho - \cos^2 \vartheta) \phi.
\end{equation}
The quantum integrability of $\Delta$ owes to the fact that this equation
 is separable. We put
\begin{equation} \label{PROD1} \phi( \rho, \vartheta) = F(\rho) G(\vartheta) , \end{equation}
and separate variables to get the coupled Mathieu equations,

\begin{equation} \label{MATHIEU} \left\{ \begin{array}{ll} \frac{{ \hbar^2}}{c^2} \, F''(\rho) +  \cosh^2\rho \, F (\rho) = \alpha F(\rho) & \rm{DBC} \; (\rm{resp.} \,NBC),
\\& \\
-\frac{\hbar^2}{c^2} \, G''(\vartheta) +  \cos^2 \vartheta \, G(\vartheta) = \alpha G(\vartheta) & \rm{PBC}. \end{array} \right.
\end{equation}
where $\hbar =  \lambda ^{-1}$ and $\alpha$ is the separation
constant. Here, `PBC' stands for `periodic boundary conditions'; DBC (resp. NBC) stands for Dirichlet (resp. Neumann) boundary conditions. Thus, we consider pairs $(\hbar, \alpha)$ where there exists
a smooth solution of the two boundary problems. 

Each of the angular and radial equations above is an eigenvalue problem for a \textit{semiclassical Schr\"odinger operator} with boundary conditions on a finite interval. These commuting operators are given by
\begin{align} \label{J}  Op_\hbar(J); \quad  J &= -{p_\rho^2}/{c^2} + \cosh^2(\rho),     \\   \label{I} Op_\hbar(I); \quad I & = {p_\vartheta^2}/{c^2} + \cos^2(\vartheta). \end{align}
The boundary conditions on $F$ take the form,
\begin{equation} \label{BP}  \;\;F(\rho_{\max}) = 0\;\;\;\;(\mbox{Dirichlet}),
\;\;\; F'(\rho_{\max})= 0\;\;\; (\mbox{Neumann}). \end{equation}
As $G(-\vartheta)$ is a solution whenever $G(\vartheta)$ is, we restrict our attention to $2 \pi$-periodic solutions to the angular equation which are either even or odd. One can then see that:

\begin{rema}\label{Smoothness between foci} In order to obtain solutions well-defined on the line segment joining the foci, i.e. at $\rho=0$, solutions to the radial equation must satisfy the boundary condition
$F'(0) = 0$ in case the solution $G$ is even and $F(0) = 0$ in case $G$ is odd. In these cases the solutions $F$ are also respectively even and odd functions.	\end{rema}
\subsection{Mathieu and modified Mathieu characteristic numbers} \label{Mathieu and modified Mathieu}For each fixed $\hbar$, the angular problem is a Sturm-Liouville problem and thus there exist real valued sequences $\{a'_n(\hbar)\}_{n=0}^\infty$ and $\{b'_n(\hbar)\}_{n=1}^\infty$ so that it has $2 \pi$-periodic non-trivial solutions - even solutions if $\alpha = a_n(\hbar)$ and odd solutions if $\alpha = b_n(\hbar)$. Here even or odd is with respect to $\vartheta \to - \vartheta$, or equivalently $y \to -y$. We represent the corresponding solutions by $G_n^{\text{e}}(\vartheta, \hbar)$ and $G_n^{\text{o}}(\vartheta, \hbar)$, respectively. The even indices correspond to $\pi$-periodic solutions, thus they must be invariant under $\vartheta \to \pi - \vartheta$, or equivalently be even with respect to $x \to -x$. Solutions with odd indices have anti-period $\pi$ and correspond to odd solutions in the $x$ variable. The sequences $a'_n(\hbar)$ and $b'_n(\hbar)$ are related to the standard \textit{Mathieu characteristic numbers of integer orders} $a_n(q)$ and $b_n(q)$ by 
\begin{equation} \label{aa'relation} a'_n(\hbar) = \frac{1}{2} + \frac{a_n(q)}{4q}, \quad  b'_n(\hbar) = \frac{1}{2} + \frac{b_n(q)}{4q}, \quad q = \frac{c^2}{4 \hbar^2}. \end{equation}
Thus using the wellknown properties of $a_n$ and $b_n$, for $\hbar >0$ we have
\begin{equation}\label{anbn} a'_0 (\hbar) < b'_1(\hbar) < a'_1 (\hbar) < b'_2(\hbar) < a'_2(\hbar) < b'_3(\hbar) <  \cdots,\end{equation}
\begin{equation} b'_{n+1}(\hbar) - a'_n(\hbar) = {\mathcal O}_n (e^{-C/\hbar}), \quad C>0.  \end{equation}  
The sequence \eqref{anbn} is precisely the spectrum of the angular Schr\"odinger operator on the flat circle $\R / (2\pi\Z)$.

Similarly for the radial problem (say with Dirichlet boundary condition $F(\rho_{\max}) = 0$), for each $\hbar$, there exist sequences $\{A'_m (\hbar)\}_{m=0}^\infty$ and $\{B'_m(\hbar)\}_{m=1}^\infty$ such that the radial problem has a non-trivial even solution $F_m^{\text{e}}(\rho,\hbar)$ if $\alpha = A'_m(\hbar)$, and a odd solution
$F_m^{\text{o}}(\rho, \hbar)$ if $\alpha = B'_m(\hbar)$. The sequences of $A'_m(\hbar)$ and $B'_m(\hbar)$ are related to \textit{modified Mathieu characteristic numbers} $A_m(q)$ and $B_m(q)$ (See \cite{Ne}) by the same relations as in \eqref{aa'relation}. They form the spectrum of the radial semiclassical Schr\"odinger operator on the interval $[- \rho_{\text{max}},  \rho_{\text{max}}]$ with Dirichlet boundary condition and satisfy
\begin{equation} A'_0 (\hbar) < B'_1(\hbar) < A'_1 (\hbar) < B'_2(\hbar) < A'_2(\hbar) < B'_3(\hbar) <  \cdots. \end{equation}

\subsection{Eigevalues of $E$: Intersection of Mathieu and modified Mathieu curves} \label{eigenvalues}

In order to find eigenfunctions of the ellipse $E$ one has to search specific values of $\hbar$ such that both radial and angular Sturm-Liouville problems possess non-trivial solutions for the same value of $\alpha$. By Remark \ref{Smoothness between foci}, we only consider the separable solutions 
$$ F_m^{\text{e}}(\rho, \hbar) G_n^{\text{e}}(\vartheta, \hbar) \quad \text{and} \quad F_m^{\text{o}}(\rho, \hbar) G_n^{\text{o}}(\vartheta, \hbar).$$ Thus the frequencies of $E$ with Dirichlet boundary condition\footnote{In the Neumann case, $A_n$ and $B_m$ are different from the ones for the Dirichlet case.} are of the form 
$$\lambda_{mn}^{\text{e}} = \frac{1}{\hbar_{mn}^{\text{e}}} \quad \text{and} \quad \lambda_{mn}^{\text{o}}=\frac{1}{\hbar_{mn}^{\text{o}}},$$
where $\hbar_{mn}^{\text{e}}$ and $\hbar_{mn}^{\text{o}}$ are, respectively, solutions to
\begin{equation}\label{Intersection} a'_n (\hbar) = A'_m(\hbar) \quad \text{and} \quad b'_n (\hbar) = B'_m(\hbar).\end{equation} 
The existence of the point of intersection of the curves $a'_n(\hbar)$ with $A'_m(\hbar)$, and $b'_n(\hbar)$ with $B'_m(\hbar)$ are guaranteed by:
\begin{theo}[Neves \cite{Ne}] For each $(m, n)$, there is a unique positive solution $q$ to each of the equations  $a_n (q) = A_m(q)$ and $b_n (q) = B_m(q)$.
	\end{theo}
Hence the same statement holds for the equations \eqref{Intersection} by the correspondence $\eqref{aa'relation}$.  The frequencies $\lambda_j$ of $E$ are obtained by sorting $\{\lambda_{mn}^{\text{e}}, \lambda_{mn}^{\text{o}}; (m, n) \in \mathbb N^2 \}$ in increasing order.

\subsection{\label{SYMCLASSES} Symmetries classes}

The irreducible representations of the $\Z_2 \times \Z_2$ symmetry group are real one-dimensional spaces, so that 
 there exists an orthonormal basis of  eigenfunctions of the ellipse  which are even or odd with respect
to each $\Z_2$ symmetry, i.e. have one of the four symmetries $$(\rm{even}, \rm{even}), (\rm{even}, \rm{odd}), (\rm{odd}, \rm{even}), (\rm{odd}, \rm{odd}),$$ where the first and the second entries correspond to symmetries with respect to $x \to -x$ and $y \to -y$, respectively. 
Given the above discussion the symmetric eigenfunctions are:
\begin{equation} \label{SeparatedEigenfunctions}\left\{ \begin{array}{llll}
(\rm{even}, \rm{even}): & \phi_{m, 2k}^{\text{e}} &= F^{\text{e}}_{m}(\rho, \hbar) G^{\text{e}}_{2k}(\vartheta, \hbar); \quad & \hbar=\hbar_{m, 2k}^{\text{e}}, \\  \\
(\rm{even}, \rm{odd}): & \phi_{m, 2k}^{\text{o}} &= F^{\text{o}}_{m}(\rho, \hbar) G^{\text{o}}_{2k}(\vartheta, \hbar); \quad  & \hbar=\hbar_{m, 2k}^{\text{o}}, \\ \\
(\rm{odd}, \rm{even}): & \phi_{m, 2k+1}^{\text{e}}& = F^{\text{e}}_{m}(\rho, \hbar) G^{\text{e}}_{2k+1}(\vartheta, \hbar); \quad  & \hbar=\hbar_{m, 2k+1}^{\text{e}},\\ \\
(\rm{odd}, \rm{odd}):  & \phi_{m, 2k+1}^{\text{o}} &= F^{\text{o}}_{m}(\rho, \hbar) G^{\text{o}}_{2k+1}(\vartheta, \hbar); \quad  &\hbar= \hbar_{m, 2k+1}^{\text{o}}. \end{array} \right. \end{equation}

\begin{figure} 
	\begin{overpic}[width=0.5\textwidth,tics=10]{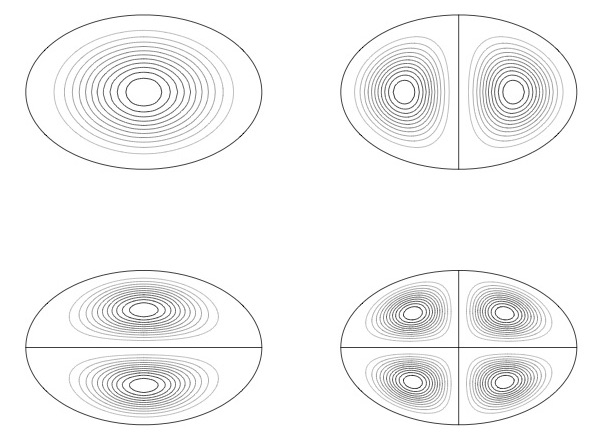}
		\put (5,40) {\small even in $x$, even in $y$}
		\put (60,40) {\small odd in $x$, even in $y$}
		\put (5, -2) {\small even in $x$, odd in $y$}
		\put (60, -2) {\small odd in $x$, odd in $y$}
	\end{overpic}
	\caption{Symmetries classes of Dirichlet eigenfunctions corresponding to the first four eigenvalues, shown by their probability densities.} \label{SymmetryClasses}
\end{figure}
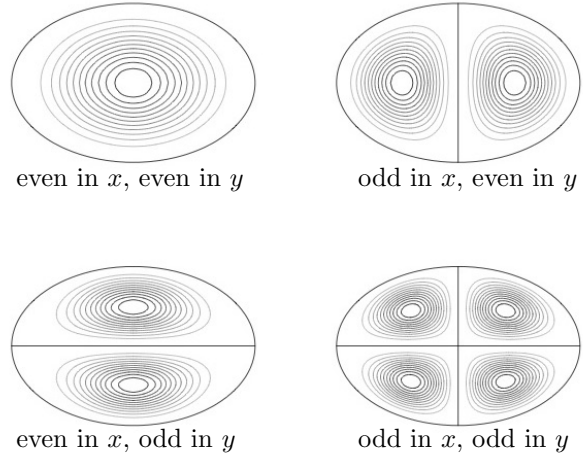

Figure \ref{SymmetryClasses} shows the symmetries classes of eigenfunctions distinguished by their probability densities. 
It is possible that two symmetric eigenfunctions correspond to the same eigenvalue, or it is possible that
they correspond to different eigenvalues.  

\subsection{Semiclassical actions and Bohr-Sommerfeld quantization conditions for the ellipse} \label{BSQC}
 Graphs of the one-dimensional classical potentials are given in \cite[Figure 1]{WWD}. The potential $-\cosh^2 \rho$ for $Op_\hbar(J)$ in \eqref{J} is a potential barrier with a single local maximim which is symmetric around the vertical line through the local maximum.  The classical potential $\cos^2 \vartheta$ underlying $Op_\hbar(I)$ in \eqref{I} is a double-well potential on the circle. Thus, there exists a separatrix curve corresponding to the two local maxima of the potential,
which divides the two-dimensional phase space into two regions. Inside the phase space curve, the level sets of the potential
are `circles' paired by the left right symmetry across the vertical line through the local maximum at $\pi$. Outside the separatrix,
the level sets have non-singular projections to the base, i.e. are roughly horizontal. 

As will be seen below, the  Bohr-Sommerfeld levels
inside the separatrix are invariant under the up-down symmetry and
have two components exchanged by the left-right symmetry. The levels
outside the separatrix are invariant under the left-right symmetry and
are exchanged under the up-down symmetry.

It is more important for our purposes to determine the lattice of semi-classical eigenvalues in terms of classical and quantum action variables.  
 The WKB (or EKB) quantization for the actions are given in \cite[(33)]{Sie97} (see also \cite{KR60} for the original reference). Up to $\mathcal O (\hbar^2)$ terms they have the form:
 $$ \text{Odd in $y$} \, \begin{cases} \begin{array}{lllll} I > 1: &  I_\rho = (m + \frac{3}{4}) \hbar, & I_\vartheta = (n + 1)\hbar, & m, n = 0,1,2, \dots, \\ &&&&\\
 I < 1: &  I_\rho = (m + 1)\hbar, & I_\vartheta = (n + \half)\hbar, & m, n = 0,1,2, \dots,   \end{array} \end{cases} $$
 $$ \text{Even in $y$}\, \begin{cases} \begin{array}{lllll} I > 1: &  I_\rho = (m + \frac{3}{4})\hbar, & I_\vartheta = n \hbar, & m, n = 0,1,2, \dots, \\&&&&\\
  I < 1: &  I_\rho = (m + \half)\hbar, & I_\vartheta = (n + \half)\hbar , & m, n = 0,1,2, \dots. \end{array} \end{cases} $$
  There is a discontinuity at $I =1$ due to the separatrix curve, but it is not important for our problem and we ignore it.



\subsubsection{Semiclassical action} In fact for each of the eight Bohr-Sommerfeld quantization condition above there is a version which is valid to all orders in $\hbar$ which are essentially given by the quantum Birkhoff normal form around each orbit under consideration. To be precise, there exist eight so called semiclassical actions
$$S_\hbar^{e/o, 1^\pm, \rho/\vartheta}(\alpha), $$
where the choices of $e/o$ corresponds to even or odd in the $y$ (equivalently in the $\vartheta$) variable, of $1^+$ or $1^-$ to $I >1$ or $I <1$, and $\rho$ or $\vartheta$ to actions in the $\rho$ or $\vartheta$ variable, respectively.  Each of the eight semiclassical actions has an $\hbar$ asymptotic expansion of the form
$$S_\hbar(\alpha) = S_0(\alpha) + \hbar S_1(\alpha) + \hbar^2 S_2(\alpha) + \cdots,$$ 
where $S_0(\alpha)$ is the corresponding classical action which is $I_\rho|_{I = \alpha}$  for $ S^{e/o, 1^\pm, \rho}_\hbar$ and $I_\vartheta|_{I=\alpha} $ for $ S_\hbar^{e/o, 1^\pm, \vartheta}$. See equations \eqref{AngularAction} and \eqref{RadialAction} for formulas for the classical actions in terms of $I=\alpha$. Then the Bohr-Sommerfeld Quantization Conditions (BSQC) to all orders are given by 
\begin{align} 
\label{BSQC outside radial} S_\hbar^{e/o, 1^+, \rho}(\alpha^{e/o, 1^+, \rho}_m(\hbar)) &=  m \hbar, \quad &\text{valid uniformly for $ \alpha \in [1+\epsilon, \cosh^2{\rho_{\max}} -\epsilon  ] $}, \\
\label {BSQC outside angular} S_\hbar^{e/o, 1^+, \vartheta}(\alpha^{e/o, 1^+, \vartheta}_n(\hbar))  &= n\hbar, \quad &\text{valid uniformly for $ \alpha \in [1+\epsilon, \cosh^2{\rho_{\max}} -\epsilon  ] $} \\  
S_\hbar^{e/o, 1^-, \rho}(\alpha^{e/o, 1^-, \rho}_m(\hbar)) &=  m \hbar, \quad &\text{valid uniformly for $ \alpha \in [\epsilon, 	1 -\epsilon  ] $}, \\
S_\hbar^{e/o, 1^-, \vartheta}(\alpha^{e/o, 1^-, \vartheta}_n(\hbar)) & = n\hbar, \quad &\text{valid uniformly for $\alpha \in [\epsilon, 1 -\epsilon  ] $}, \end{align}
where $\epsilon >0$ is arbitrary, however the remainder estimates in the asymptotic expansions depend on $\epsilon$. There are versions of BSQC in the literature that are valid uniformly near the separatrix but we do not need it here. 
We also point out that the Maslov indices are not ignored but absorbed in the corresponding subleading terms $S_1(\alpha)$. 

\begin{rema} By our notations of Section \ref{Mathieu and modified Mathieu} on the Mathieu and modified Mathieu characteristic values, away from the separatrix level we have,
	$$ \{ \alpha^{e, 1^\pm, \rho}_m(\hbar); \; m=0, 1, 2, \cdots \}  = \{ A_m'(\hbar): \; m =0, 1, \cdots \}, $$
		$$ \{ \alpha^{o, 1^\pm, \rho}_m(\hbar); \;  m=0, 1, 2, \cdots \}  = \{ B'_m(\hbar): \; m =1, 2, \cdots \}, $$
			$$ \{ \alpha^{e, 1^\pm, \vartheta}_n(\hbar); \; n=0, 1, 2, \cdots \}  = \{ a'_n(\hbar): \; n =0, 1, \cdots \}, $$
				$$ \{ \alpha^{o, 1^\pm, \vartheta}_n(\hbar); \; n=0, 1, 2, \cdots \}  = \{ b'_n(\hbar): \; n =1, 2, \cdots \}. $$
	\end{rema}
The eigenvalues of $E$ are determined by intersecting the above analytic curves as follows:
\begin{equation}\label{Intersection of alpha}  \alpha^{e, 1^\pm, \rho}_m(\hbar) =   \alpha^{e, 1^\pm, \vartheta}_n(\hbar), \end{equation}
$$ \alpha^{o, 1^\pm, \rho}_m(\hbar) =   \alpha^{o, 1^\pm, \vartheta}_n(\hbar),$$
the solutions of which are precisely $\hbar^{e}_{mn}$ and $\hbar^{o}_{mn}$, respectively, that we introduced in Section \ref{eigenvalues}.

\subsection{\textbf{Keller-Rubinow algorithm}} \label{Keller-Rubinow algorithm}  In this section we explore the procedure of finding $\hbar^e_{mn}$ corresponding to eigenvalues associated to invariant curves outside the separatrix (i.e. $1^+$ case) whose eigenfunctions are even in the $\vartheta$ variable. All other cases follow a similar procedure and we shall drop the superscripts for convenience. 

We are in search of solutions to equation \eqref{Intersection of alpha} which, in our convenient notation, are given by
\begin{equation} \label{alphaeq} \alpha^{\rho}_m(\hbar) =   \alpha^{\vartheta}_n(\hbar), \end{equation}
where the left and the right hand sides satisfy the BSQC \eqref{BSQC outside radial} and \eqref{BSQC outside angular},
\begin{equation} \label{Angular and Radial BSQC}S_\hbar^{\rho}(\alpha^{ \rho}_m(\hbar)) =  m \hbar, \quad S_\hbar^{\vartheta}(\alpha^{ \vartheta}_n(\hbar)) = n\hbar, \end{equation} respectively. 
Following \cite{KR60}, we divide these two equations to obtain,
\begin{equation} \label{Equation of intersection} A_\hbar(\alpha): =  \frac{S_\hbar^{\rho}(\alpha)}{ S_\hbar^{\vartheta}(\alpha)}  = \frac{I_\rho(\alpha) -\frac{3}{4} \hbar + \sum_{k=2}^\infty S^\rho_k(\alpha) \hbar^k }{I_\vartheta(\alpha) +  \sum_{k=2}^\infty S^\vartheta_k(\alpha) \hbar^k} = \frac{m}{n}. \end{equation}
The expression $A_\hbar(\alpha)$ has a classical $\hbar$ expansion with principal term 
\begin{equation} \label{A_0} A_0(\alpha):= \frac{I_\rho(\alpha)}{I_\vartheta (\alpha)}, \end{equation}
which is a positive monotonic function on the interval $[1, \cosh^2 \rho_{\max}]$ (See \cite{KR60}, page 41). Hence, if we choose $r$ in the range of $A_0(\alpha)$ on the domain $[1+2\epsilon, \cosh^2 \rho_{\max} -2 \epsilon]$, then for $\hbar$ sufficiently small there is a unique solution $\alpha$ to the equation $A_\hbar(\alpha)=r$ in the slightly larger interval $[1+\epsilon, \cosh^2 \rho_{\max} -\epsilon]$, accepting an $\hbar$ expansion of the form:
\begin{equation} \label {alpha h r} \alpha \left (\hbar, r \right ) = \sum_{k=0}^\infty \alpha_{(k)} (r) \hbar^k.\end{equation}
 It is manifestly the the inverse function of $A_h(\alpha)$ and its formal power series coefficients $\alpha_{(k)}(r)$ are smooth functions of $r$. The principal term $\alpha_{(0)}$ is the inverse function of $A_0(\alpha)$. By this definition, the solution to \eqref{Equation of intersection} is $ \alpha ( \hbar, m/n)$ whenever $m/n$ belongs to  $A_0 [1+2\epsilon, \cosh^2 \rho_{\max} - 2\epsilon]$, which is a  bounded  closed interval in $(0, \infty)$. In particular $m/n$ is bounded above and below by positive constants $K_1$ and $K_2$:
 \begin{equation} \label{sector}(m, n) \in \mathbb N^2: \quad K_1 \leq \frac{m}{n}  \leq K_2. \end{equation}
 This is the eligible sector of lattice points for our eigenvalue problem outside the separatrix. 
  Plugging $\alpha(\hbar, m/n)$ into the angular BSQC, i.e. the second equation of \eqref{Angular and Radial BSQC}, (the radial one follows immediately from the angular one and \eqref{Equation of intersection}), we arrive at the quantization condition for the eigenvalues of $E$:
\begin{equation} Q(\hbar, m, n):=  \frac{1}{n} S^\vartheta_\hbar ( \alpha(\hbar, m/n) ) = \hbar. \end{equation}
We claim that for $m$ and $n$ sufficiently large, this equation has a unique solution $\hbar_{mn}$ in a sufficiently small interval $[0, \hbar_0]$, or equivalently the function $Q( \cdot, m, n)$ has a unique fixed point.   Now, since
$$ Q(0, m, n) = \frac{I_\vartheta(\alpha_{(0)}(m/n))}{n} , \quad \frac{\d Q}{\d \hbar}(0, m, n) = 0,$$
for $\hbar_0$ sufficiently small, and $n$ sufficiently large $Q( \cdot , m, n)$ maps $[0, \hbar_0]$ into itself and $\frac{\d Q}{\d \hbar}(\hbar, m, n) <\frac12$ in this interval. The claim follows by the Banach contraction principle.

\begin{rema} \label{alpha at rationals}
Since there are many functions $\alpha$ used, it is important to highlight their relations and differences. If we evaluate $\alpha \left (\hbar, r \right )$ defined in \eqref{alpha h r}, at $\hbar = \hbar_{m,n}$ and $r = \frac{m}{n}$, we get the common value of \eqref{alphaeq}. In short,
$$\alpha \left (\hbar_{mn}, \frac{m}{n} \right ) =  \alpha^{\rho}_m(\hbar_{mn}) =   \alpha^{\vartheta}_n(\hbar_{mn}).$$
We also note that the function $\alpha_{(0)}(r)$, with parentheses around $0$,  is the principal term of $\alpha(\hbar, r)$ and should not be confused with $\alpha^\rho_0(\hbar)$ or $\alpha_0^\vartheta(\hbar)$. 
\end{rema}

 In fact, the above procedure provides an asymptotic for $\lambda_{mn} = 1/ \hbar_{mn}$ and
gives a sharper result than previously known:
\begin{prop} \label{INVCURVESOUT} The frequencies $\lambda^{e/o}_{mn}$ of $E$ associated to invariant curves outside the separatrix curve, and $\epsilon$ away from it, correspond to lattice points $(m, n) \in \mathbb N^2$ in the sector
	$$ \min \left \{ \frac{I_\rho(\alpha)}{I_\vartheta(\alpha)}; \; \alpha \in [1+\epsilon, \cosh^2 \rho_{\max} - \epsilon] \right \}    \leq \frac{m}{n}  \leq  \max \left \{ \frac{I_\rho(\alpha)}{I_\vartheta(\alpha)}; \; \alpha \in [1+\epsilon, \cosh^2 \rho_{\max} - \epsilon] \right \} ,$$
	and satisfy the asymptotic property,
	$$ \lambda^{e/o}_{mn} =\frac{n}{I_\vartheta(\alpha_{(0)}(m/n))} + \mathcal O\left (\frac{1}{n} \right).$$
	The same asymptotic formula holds for the frequencies $\lambda^{e/o}_{mn}$  associated to invariant curves inside the separatrix curve, except in this case the sector of lattice points is:
	$$ \min \left \{ \frac{I_\rho(\alpha)}{I_\vartheta(\alpha)}; \; \alpha \in [\epsilon, 1 - \epsilon] \right \}    \leq \frac{m}{n}  \leq  \max \left \{ \frac{I_\rho(\alpha)}{I_\vartheta(\alpha)}; \; \alpha \in [\epsilon, 1- \epsilon] \right \} ,$$

\end{prop} 
The effects of even/odd are only reflected in the remainder term $\mathcal O (1/n)$, which in addition depends on the distance $\epsilon$ from the separatix. Note that the explicit formulas for  $I_\vartheta$ and $I_\rho$ (hence for $\alpha_{(0)}$) in terms of elliptic integrals are different for the inside and outside the separatrix curve (See for example \cite{Sie97}).

\section{Localization of boundary values of separable eigenfunctions on invariant curves. Proof of Theorem \ref{LOCAL}}

In this section, we relate semi-classical asymptotics of eigenfrequencies $\lambda^{\text{e/o}}_{m, n} = 1 / \hbar^{\text{e/o}}_{m, n}$ and of the associated separated eigenfunctions $\phi^{\text{e/o}}_{m, n}$ defined by \eqref{SeparatedEigenfunctions} along `ladders' or `rays' in the action lattice $(m, n) \in \mathbb N^2$. In particular, different rays correspond to different invariant Lagrangrian
submanifolds for the billiard flow. It is simpler to use the billiard map and then to relate rays in the joint spectrum to invariant curves
for the billiard map. Given an invariant curve, inside or outside the separatrix, we wish to find a ray in the joint spectrum for which
the associated eigenfunctions concentrate on the curve. 
Since the WKB method is highly developed in dimension one, it suffices for our purposes to locate the ray in $\mathbb N^2$ which corresponds
to the invariant curve.  The corresponding eigenfunctions will then concentrate on the corresponding Lagrangian submanifolds.

\begin{prop} \label{modified u} Let $\phi^{e/o}_{m, n}(\rho, \vartheta)$  be a separable Dirichlet (resp.  Neumann)
eigenfunction defined in \eqref{SeparatedEigenfunctions}.  Then the `modified boundary
trace'
$${u_{m,n}^{e/o}}(\vartheta)= \left\{ \begin{array}{ll}  \phi^{e/o}_{m, n}(\rho, \vartheta)|_{\rho =\rho_{\max}},
& \mbox{Neumann}, \\ & \\
\frac{1}{\lambda_{mn}^{e/o}}  \frac{\partial \phi^{e/o}_{m, n}(\rho, \vartheta)}{\partial \rho}|_{\rho =\rho_{\max}}, & \mbox{Dirichlet}.
\end{array} \right.$$ is an eigenfunction of the angular Schr\"odinger operator
 $\{Op_\hbar(I)\}_{\hbar =\hbar^{\text{e/o}}_{m, n}} $, whose eigenvalue $\alpha$ is determined by
\begin{equation} \label{alphaform} \frac{\langle Op_\hbar(I) u_{m,n}^{e/o} , u_{m,n}^{e/o}
\rangle_{L^2(\partial E)}}{\langle u_{m,n}^{e/o} , u_{m,n}^{e/o} \rangle_ {L^2(\partial E)}}, \end{equation} 
which is $ \alpha^{e/o, 1^+, \vartheta}_n(\hbar)$ if it is $>1$ and  $\alpha^{e/o, 1^-, \vartheta}_n(\hbar)$ if it is $<1$. 
\end{prop}
\begin{proof} The proof is obvious by equations \eqref{SeparatedEigenfunctions}, \eqref{MATHIEU}, and \eqref{I}. \end{proof} 

\begin{rema} \label{modified u vs u}
It is important to note that although in the Neumann case our modified boundary trace ${u_{m,n}^{e/o}}$ is the same as the boundary trace $\left({u_{m,n}^{e/o}}\right)^b$ defined by \eqref{ujbdef}, but they are slightly different in the Dirichlet case as in this case
$$  \left({u_{m,n}^{e/o}}\right)^b = - \frac{1}{ \sqrt{c^2(\cosh^2 \rho_{\max} - \cos^2 \vartheta)}}{u_{m,n}^{e/o}}, $$
which is due to the relation 
$$  \frac{\d}{\d \nu} = - \frac{1}{ \sqrt{c^2(\cosh^2 \rho_{\max} - \cos^2 \vartheta)}}\frac{\d}{\d \rho} \bigg\rvert_{\rho = \rho_{\max}}.$$
\end{rema}

Our goal is to show that, for any invariant curve $I = \alpha$, of the billiard map lying inside or outside the separatrix curve, there exists a ladder of separable eigenfunctions $\phi^{e/o}_{m, n}$ whose Cauchy data $ \left({u_{m,n}^{e/o}}\right)^b$  concentrates on the invariant curve in $B^* \partial E$. In order to prove this we first need the following lemma.

\begin{lem}\label{SEQ}  For any $\alpha \in [0, \cosh^2 \rho_{\max}]$, there exists a subsequence of $\{\hbar^{e/o}_{mn}: \; (m, n) \in \mathbb{N}^2 \}$ 
	(for either Dirichlet or Neumann boundary conditions) along which the eigenvalues of the semiclassical angular operator $\{Op_\hbar (I)\}|_{\hbar = \hbar^{e/o}_{mn}}$ converges to $\alpha$. Here, $e/o$ means that any choice of even or odd can be selected. 
\end{lem}

\begin{proof} It suffices to prove that
	\begin{itemize} 
		\item[(1)] For any $\alpha \in (1, \cosh^2 \rho_{\max})$ corresponding to invariant curves outside the separatrix, there exists a subsequence of $\{\hbar^{e/o}_{mn}: \; (m, n) \in \mathbb{N}^2 \}$ 
(for either Dirichlet or Neumann boundary conditions) along which
$$\alpha^{e/o, 1^+, \vartheta}_n(\hbar^{e/o}_{mn}) \to \alpha. $$ 
\item[(2)] For any $\alpha \in (0, 1)$ corresponding to invariant curves inside the separatrix, there exists a subsequence
along which
$$\alpha^{e/o, 1^-, \vartheta}_n(\hbar^{e/o}_{mn}) \to \alpha.$$ 
\end{itemize}
A density argument would take care of the levels $\alpha=0$, $1$ and $\cosh^2 \rho_{\max}$. 

We shall only prove (1), as the proof of (2) is similar. Furthermore, we shall only focus on the even case because the  proof for the odd case is identical. Fix $\alpha \in (1, \cosh^2 \rho_{\max})$. We choose $\epsilon >0$ so that $\alpha \in [1+2\epsilon, \cosh^2 \rho_{\max}- 2\epsilon]$.  Let $\hbar_{mn}$ be the sequence we found in Section \ref{Keller-Rubinow algorithm} associated to the level curves outside the separatrix and to even eigenfunctions (even in the $y$ variable). By Remark \ref{alpha at rationals}, it suffices to show that there is a subsequence $(m_j, n_j)$ along which
$$ \alpha \left( \hbar_{m_j, n_j}, \frac{m_j}{n_j} \right) \to \alpha, \quad (j \to \infty). $$ We choose  $r_0$ by $\alpha_{(0)}(r_0) = \alpha$ (recall that $\alpha_{(0)}$ is monotonic) and choose a sequence of lattice points $(m_j, n_j) \in \mathbb N ^2$ in the eligible sector \eqref{sector} such that $\frac{m_j}{n_j} \to r_0$ and $|(m_j, n_j)| \to \infty$. Since, 
$$ \left| \alpha \left (\hbar_{m_j, n_j}, \frac{m_j}{n_j} \right) - \alpha_{(0)}\left (\frac{m_j}{n_j} \right ) \right|  = \mathcal O \left ( \hbar_{m_jn_j} \right) = \mathcal O \left ( {n_j^{-1}} \right ), $$
the lemma follows by letting $ j\to \infty$ and using the continuity of $\alpha_{(0)}$. 

\end{proof}

\subsection{Quantum limits of Cauchy data and the proof of Theorem \ref{LOCAL}}

By Proposition \ref{modified u}, the modified boundary
traces ${u_{m,n}^{e/o}}(\vartheta)$ of the separable eigenfunctions $\phi^{e/o}_{m, n}(\rho, \vartheta)$ of $\Delta$, 
are eigenfunctions of the semiclassical angular Schr\"odinger operator $\{ Op_\hbar(I) \}_{\hbar=\hbar^{e/o}_{mn}}$.
  It is well-known that eigenfunctions
of $1D$ semi-classical Schr\"odinger operators localize on level
sets of the symbol. Thus if we fix $\alpha$ in the action interval and choose a sequence of $\{\hbar^{e/o}_{mn}\}$ provided by Lemma \ref{SEQ}, then we know that along this sequence the quantum limit of $|{u_{m,n}^{e/o}} |^2 d \vartheta$ is a measure on $B^*\partial E$ that is supported on $I = \alpha$. We also know, by Egorov's theorem, that this measure must be invariant under the flow of $I$, therefore the quantum limit must be the Leray measure $d \mu_\alpha$. Since, by Remark \ref{modified u vs u}, in the Dirichlet case the boundary traces $\left (u_{m,n}^{e/o} \right )^b$ differ from $u_{m,n}^{e/o}$ by a factor $\left( c^2(\cosh^2 \rho_{\max} - \cos^2 \vartheta) \right ) ^{-1/2}$ caused by the conformal transformation from Cartesian to elliptical coordinates, and since 
$$ds = \sqrt{c^2(\cosh^2 \rho_{\max} - \cos^2 \vartheta)} d \vartheta,$$
we get
$$\left| \left (u_{m,n}^{e/o} \right )^b \right |^2 ds =  \frac{1}{\sqrt{c^2(\cosh^2 \rho_{\max} - \cos^2 \vartheta)}} \left| u_{m,n}^{e/o} \right |^2  d \vartheta \rightarrow  \frac{d \mu_\alpha}{\sqrt{c^2(\cosh^2 \rho_{\max} - \cos^2 \vartheta)}}, $$ which proves Theorem \ref{LOCAL} in the Dirichlet case. The Neumann case is essentially the same; we omit the details.

\section{Hadamard variational formulae for isospectral deformations}

We consider the Dirichlet (resp. Neumann) eigenvalue problems for
a one parameter family of Euclidean plane domain $\Omega_t$, where
$\Omega_0 = E$ is an ellipse:
\begin{equation} \label{EIG}  \left\{ \begin{array}{ll} - \Delta \phi_{j}(t)  =
\lambda_j^2(t) \phi_{j}(t) \;\; \mbox{in}\;\; \Omega_t, & \;\;
\\ \\\phi_{j}(t)
 = 0 \;\;(\mbox{resp.} \; \partial_{\nu_t} \phi_{j}(t) = 0) \; \mbox{on} \;\; \partial \Omega_t.
\end{array} \right. \end{equation} Here,
$\partial_{\nu_t}$ is the interior unit normal to $\d \Omega_t$. When
$\lambda_j^2(0)$ is a simple eigenvalue, then under a $C^1$
deformation the eigenvalue moves in a $C^1$ curve $\lambda^2_j(t)$.
When $\lambda^2_j(0)$ is a multiple eigenvalue, then in general the eigenvalue may split into branches. 
Examples in \cite{Kato} show that eigenfunctions do not necessarily deform nicely if the deformation is not analytic. Hence we cannot even assume that eigenfunctions are $C^1$ if the deformation is only $C^1$. 
  However, we  assume in this section  that
the deformation is isospectral. In this case, a multiple eigenvalue does not change multiplicity under the deformation, and therefore
there is no splitting into branches. 

When an eigenvalue has multiplicity $> 1$, there exists an  orthonormal basis (known as the Kato-Rellich basis)
of the eigenspace which moves smoothly under the deformation. The multiple eigenvalue splits under a generic perturbation and one can only expect a perturbation formula along each path. 
When we assume  that the deformation is isospectral, hence that the  eigenvalue does not split (or even
change) along the deformation, then there exists a Kato-Rellich basis for the eigenspace. 
\subsection{\label{HD} Hadamard variational formulae}

As in the
introduction, we parameterize the deformation by a function
$\rho_t$ on $\partial E$ so that $\partial \Omega_t$ is the graph
of $\rho_t$ over $\partial \Omega_0 = \partial E$ in the sense
that $\partial \Omega_t = \{x + \rho_t(x) \nu_x: x \in \partial
\Omega_0\}$. If $\dot{\rho} : = \frac{d}{dt} \rho_t |_{t = 0}
\not= 0$, then the first order variation of eigenvalues is the
same as for the deformation by $x + t \dot{\rho}(x) \nu_x$.  In this section we review the Hadamard variational formula in the case of simple eigenvalues. We refer to \cite[Section 1]{HeZ12} for background on the Hadamard variational formula.

 When $\lambda_j^2(0)$ is a simple eigenvalue (i.e. of
 multiplicity one) with $L^2$-normalized eigenfunction $\phi_j$, then Hadamard's variational formula for plane domains is that
\begin{equation} \label{DOT} \mbox{Dirichlet:}\;\;\;\; {(\lambda_j^2)}^{\cdot} = \int_{\partial \Omega_0} ( \partial_{\nu}
\phi_{j})^2  \dot{\rho}\, ds, \end{equation} where $ds$
is the induced arc-length measure.  Hence, under an infinitesimal
isospectral deformation we have, for every simple eigenvalue,
\begin{equation} \label{DHD}\mbox{Dirichlet:}\;\;\;\; \int_{\partial \Omega_0} (\partial_{\nu}
\phi_{j})^2  \dot{\rho} \, ds = 0. \end{equation}

Hadamard's variational formula is actually a variational formula
for the variation of the  Green's functions $G(\lambda,
x, y)$ with the given boundary conditions. In the Dirichlet case
it states that
$$ \dot{G} (\lambda, x, y) = - \int_{\partial \Omega_0}  \frac{\partial}{\partial \nu_1} G(\lambda, q, x)
\frac{\partial}{\partial \nu_1} G(\lambda, q, y) \dot{\rho} ds. $$
The formula (\ref{DHD}) follows if we compare the poles of order
two on each side.  The same comparison shows that if the
eigenvalue $\lambda^2_j(0)$ is repeated with multiplicity $m(\lambda_{j}(0))$ and if $\{\lambda_{j k}(t)\}_{j = 1}^{m(\lambda_j(0))}$ is
the perturbed set of eigenvalues, then
$$\frac{d}{dt} \bigg\rvert_{t=0} \sum_{k = 1}^{m(\lambda_j(0))} \lambda_{j k}^2(t) = \sum_{k =
1}^{m(\lambda_{j}(0))} \int_{\partial \Omega_0} (\partial_{\nu}
\phi_{j, k})^2   \dot{\rho} \, ds. $$
Here $\{\phi_{j, k} \}_{j = 1}^{m(\lambda_j(0))}$ is any ONB of the repeated eigenvalue $\lambda^2_j(0)$.

There exist similar Hadamard variational formulae in  the Neumann
case. When the eigenvalue is simple, we have
$$\left ( {\lambda}_j^2 \right)^\cdot = \int_{\partial \Omega_0} \left(|\nabla_{\partial \Omega_0} (\phi_{j})|^2
- \lambda_j^2 \phi_{j}^2 \right) \dot{\rho} \, ds,
$$ hence
\begin{equation} \label{NHD} \mbox{Neumann:}\;\;\; \int_{\partial \Omega_0} \left(|\nabla_{\partial \Omega_0} (\phi_{j})|^2
 - \lambda_j^2 \phi_{j}^2 \right) \dot{\rho} \,
ds= 0. \end{equation}

\subsection{Hadamard variational formula for an isospectral deformation} 

We now assume that the deformation is isospectral. As mentioned above, there exists a  Kato-Rellich basis which moves smoothly under the deformation. 
In fact, we show that for an isospectral deformation every eigenfunction has a smooth deformation along the path. In the following $-\Delta_t$ denotes the Dirichlet (resp. Neumann) Laplacian on $\Omega_t$. 

\begin{lem} \label{deformation of phi} Suppose that $\Omega_t$ is a $C^1$ Dirichlet (resp. Neumann) isospectral deformation. Then any eigenfunction $\phi_j(0)$ of $-\Delta_0$ on $\Omega_0$, has a $C^1$ deformation $\phi_j(t)$ of eigenfunctions of $-\Delta_t$ on $\Omega_t$. 
	\end{lem}
\begin{proof} 
	Let $\lambda_j^2 (0)$ be the eigenvalue of $\phi_j(0)$, of multiplicity $m_j \geq 1$, and $\gamma$ be a circle in $\C$ centered at $\lambda_j^2 (0)$ such that no other eigenvalues of $-\Delta_{0}$ are in the interior of $\gamma$ or on $\gamma$. We define
	$$P_t= -\frac{1}{2\pi i} \int_\gamma z {R}_{t}(z) \,dz,$$
	where ${R}_{t}(z)= ( -\Delta_t - z)^{-1}$ is the resolvent of $-\Delta_t$.  By the Cauchy integral formula, it is clear that $P_0$ is the orthogonal projector onto the eigenspace of $\lambda_j^2 (0)$. Since the eigenvalues $\{\lambda_{j,k}^2(t)\}_{k=1}^{m_j}$ vary continuously in $t$, for $t$ small these are the only eigenvalues of $-\Delta_t$ in $\gamma$. Therefore, in general, $P_t$ is the total projector (the direct sum of projectors) associated with $\{\lambda_{j,k}^2(t)\}_{k=1}^m$. The operator $P_t$ is $C^1$ in $t$, since  the resolvent (hence,  Green's function) is $C^1$ in $t$ (see  \cite[Theorem II.5.4]{Kato}). Now assume $\Omega_t$ is an isospectral deformation. Since  the spectrum is constant along the deformation, $P_t$ projects every function on $\Omega_t$ onto an eigenfunction of $\Omega_t$ of eigenvalue $\lambda^2_j(0)$. Let $f_t$ be a $C^1$ family of smooth diffeomorphisms from $\Omega_t$ to $\Omega_0$ with $f_0 = \text{Id}$.   Then 
	$$\phi_j(t): = P_{t}( f_t^*\left ( \phi_j(0) \right )), \quad \left ( \text{here},\; f_t^*\left ( \phi_j(0) \right ) = \phi_j(0) \circ f_t \; \right)$$ must be an eigenfunction of $-\Delta_{t}$ of eigenvalue $ \la^2_j(0)$. 
	\end{proof}

We are now in position to prove:

\begin{lem} \label{HDVAR} Suppose that $\Omega_t$ is a $C^1$ isospectral deformation. Then
for any eigenfunction $\phi_j$ of $\Omega_0$,
\begin{equation} \label{Hadamard DN} \left\{ \begin{array}{ll} \int_{\partial \Omega_0} \dot{\rho} |\partial_{\nu} \phi_j|^2 = 0, & \rm{Dirichet}\\&\\
 \int_{\partial \Omega_0} \left(|\nabla_{\partial \Omega_0} (\phi_{j})|^2
 - \lambda_j^2 \phi_{j}^2 \right) \dot{\rho} \,
ds= 0, & {\rm Neumann} \end{array} \right. \end{equation}
\end{lem}

\begin{proof}

Let $\phi_j(0)$ be any eigenfunction of $\Omega_0$ and $\phi_j(t)$ be the $C^1$ deformation of eigenfunction of $\Omega_t$ provided by Lemma \ref{deformation of phi}.
For $t > 0$, the eigenvalue problem for the isospectral deformation is pulled back to $\Omega_0$ by a $C^1$ family diffeomorphisms $f_t$, with $f_0 = \text{Id}$, and has the form,
$$( \widetilde{\Delta}_{t} + \lambda^2_j) \widetilde{\phi}_j(t) = 0, $$
where $ \widetilde{\Delta}_{t}$ and $\widetilde{\phi}_j(t)$ are the pullbacks of $ {\Delta}_{t}$ and ${\phi}_j(t)$ to $\Omega_0$, respectively. 
Taking the variation gives 
$$\dot{\widetilde{\Delta}} \phi_j(0) +(\Delta_{0} + \lambda^2_j) \dot{\widetilde{\phi}_j}(0) = 0. $$
Take the inner product with $\phi_k(0)$ in the same eigenspace. Integration by parts in the second term kills the second term. Thus we
get 
$$\langle \dot{\Delta}  \phi_j(0), \phi_k(0) \rangle  = 0. $$
The variation $\dot{\Delta} $ can be calculated (see  for example \cite{HeZ12}) to obtain:
$$\int_{\partial \Omega_0} \dot{\rho} (\partial_{\nu} \phi_j) (\partial_{\nu} \phi_k) ds = 0,$$
for all $\phi_j, \phi_k$ in the $\lambda_j$-eigenspace of the Dirichlet problem. A similar proof works for the relevant quadratic form for the Neumann problem.

\end{proof}

\section{Proof of Theorem \ref{RIGID}}

Before we prove our main theorem, we need to study the limits of the equations \eqref{Hadamard DN}
along sequences of eigenvalues introduced in Theorem  \ref{LOCAL}. 

\begin{cor}\label{ZERO} Let $\dot{\rho}$ be the first variation of a Dirichlet (or Neumann) isospectral
deformation of an ellipse $E$.  Then for all $ 0\leq \alpha \leq \cosh^2(\rho_{\text{max}})$,
$$  \int_{I =\alpha}  \frac{\dot{\rho} }{\sqrt{\cosh^2 \rho_{\max} - \cos^2 \vartheta}} d \mu_\alpha = 0. $$
\end{cor}

\begin{proof} 
The Dirichlet case follows immediately from Theorem \ref{LOCAL} and Lemma \ref{HDVAR}. For the Neumann case, we observe that  by Theorem \ref{LOCAL} the quantum limit of 
$$ \lambda_j^{-2}  |\nabla_{\partial \Omega_0} (\phi_{j})|^2
- \phi_{j}^2, $$
along a sequence of eigenfunctions that concentrates on the invariant curve $I =\alpha$ is  
$$(| \eta|^2 -1) d \mu_\alpha.$$
Therefore, in the Neumann case we get
\begin{equation} \label{Integral Neumann} \int_{I =\alpha} (| \eta|^2 -1)\, \sqrt{c^2(\cosh^2 \rho_{\max} - \cos^2 \vartheta)}\,  \dot{\rho} \, d \mu_\alpha =0. \end{equation}
We recall that $\eta$ is the symplectic dual of the arclength variable $s$. From the equation $\eta ds = p_\vartheta d \vartheta$, we find that in the $(\vartheta, p_\vartheta)$ coordinates, $\eta$ is given by $$\eta = \frac{p_\vartheta}{\sqrt{c^2(\cosh^2 \rho_{\max} - \cos^2 \vartheta)}}.$$
Since on $I =\alpha$, $p_\vartheta^2 = c^2 (\alpha - \cos^2 \vartheta)$,
$$|\eta|^2 -1 =  \frac{\alpha - \cosh^2 \rho_{\max} }{\cosh^2 \rho_{\max} - \cos^2 \vartheta}.$$
The corollary follows in the Neumann case by taking out the constant $\alpha - \cosh^2 \rho_{\max}$ from the integral \eqref{Integral Neumann}.
 \end{proof}

Theorem \ref{RIGID}, now reduces to:
\begin{prop}\label{RHODOT}  The only $\Z_2 \times \Z_2$ invariant function
	$\dot{\rho}$ satisfying the equations of Corollary \ref{ZERO} is
	$\dot{\rho} = 0$ for $\alpha \in (0, 1)$, i.e. for levels inside the separatrix.  Similarly, the same statement holds if we only know equations of Corollary \ref{ZERO} for  $\alpha \in (1, \cosh^2 \rho_{\max})$, i.e. levels outside the separatrix. \end{prop}

\begin{proof} Since $\dot {\rho}(\vartheta) $ is $\mathbb Z_2 \times \mathbb Z_2$ invariant we can put 
	$$P(\cos^2 \vartheta): = \frac{\dot{\rho}(\vartheta)}{\sqrt{\cosh^2 \rho_{\max} - \cos^2 \vartheta}}.$$
	 By our explicit formula \eqref{Leray} for the Leray measure $d \mu_\alpha$, and by the $\mathbb Z_2 \times \mathbb Z_2$ symmetry, we have
$$\int_{0}^{\frac{\pi}{2}} \frac{P(\cos^2 \vartheta)} {\sqrt{(\alpha - \cos^2
\vartheta)_+} }  d \vartheta = 0, \;\; \forall \,\,
0 \leq  \alpha \leq \cosh^2 \rho_{\max}.$$ 
Splitting this equation into $\alpha \leq 1$ and $\alpha \geq 1$ cases, we obtain:
\begin{equation} \label{Inside=0} \int_{0}^{\cos^{-1}( \sqrt{\alpha}) }  \frac{P(\cos^2 \vartheta)}{ \sqrt{\alpha -  \cos^2 \vartheta}}\;  \;d \vartheta = 0, \;\; \forall \, \, 0 \leq \alpha \leq 1. \end{equation}
\begin{equation} \label{Outside=0}\int_{0}^{\frac{\pi}{2}}  \frac{P(\cos^2 \vartheta)}{ \sqrt{\alpha -  \cos^2 \vartheta}}\;  \;d \vartheta = 0, \;\;
\forall \, \, 1 \leq \alpha \leq \cosh^2 \rho_{\max}. \end{equation}

It is sufficient to show that $P \equiv 0$, given \eqref{Inside=0} or \eqref{Outside=0}. 
\\

\textbf{Proof using invariant curves inside the separatrix.} We change variables to $u = \cos \vartheta$ and also set $x =\sqrt{\alpha}$.  Then the integral \eqref{Inside=0} becomes
\begin{equation} \int_{0}^{ x}  \frac{P(u^2)}{\sqrt{x^2 -  u^2}}\;  \; \frac{du}{ \sqrt{1 - u^2}} = 0, \;\;
\forall \, \, 0 \leq x \leq 1. \end{equation} Writing $f(u) =  \frac{P(u^2)}{
\sqrt{1 - u^2}}, $ this becomes
\begin{equation} \int_{0}^{x}  \frac{f(u)}{ \sqrt{x^2 -  u^2}}\; du = 0, \;\;
\forall \, \, 0 \leq x \leq 1. \end{equation}

The transform
$$\acal f(x) = \int_{0}^{ x}  \frac{f(u)}{ \sqrt{x^2 -  u^2}}\ du
$$ is closely related to the Abel transform. We claim that the left inverse Abel transform is given
by,
$$\acal^{-1} g(u)  = \frac{2}{\pi} \frac{d}{du}   \int_{0}^{ u}  \frac{x g(x)}{ \sqrt{u^2 -  x^2}}\; dx . $$
The key point is the integral identity,
$$I(u, v) : = \int_{v}^u \frac{x dx }{\sqrt{u^2 - x^2} \sqrt{x^2 - v^2}}     = \frac{\pi}{2}, \qquad ( v \leq u).
 $$
It follows that if $\bcal g(u) $ is the integral in the purported
inversion formula,
$$\begin{array}{lll} \bcal \acal f(u) & = & \frac{2}{\pi} \frac{d}{du}   \int_{0}^{ u}  \frac{x \acal f(x)}{ \sqrt{u^2 -  x^2}}\;
dx\\ && \\
& = &
 \frac{2}{\pi} \frac{d}{du}   \int_{0}^{ u}  \int_0^x  \frac{x }{ \sqrt{u^2 -
 x^2}}\frac{f(v)}{ \sqrt{x^2 -  v^2}}\; dv dx \\ && \\
 & = & \frac{2}{\pi}\frac{d}{du}   \int_{0}^{ u} I(u, v) f(v) dv
 \\ && \\
 & = & \frac{d}{du}   \int_{0}^{ u} f(v) dv= f(u). \end{array}$$

Since $\acal$ is left invertible, it follows that $\ker \acal =
\{0\}$. Since $f(u) =  \frac{P(u^2)}{
	\sqrt{1 - u^2}}$ lies in its kernel, we have $P
= 0$ and hence $\dot{\rho}=0$.

\textbf{Proof using invariant curves outside the separatrix.} The proof of the second assertion of Proposition \ref{RHODOT}  is similar to the final steps in the proofs of spectral rigidity results of \cite{GM}, \cite{HeZ12}, and \cite{Vi}, for the ellipse in various settings. We need to show that \eqref{Outside=0} implies $P=0$.  We change  variables by $u = \cos^2 \vartheta$ and this time we set $f(u) =  \frac{P(u)}{
	\sqrt{u(1 - u)}}$. Then
$$\int_{0}^{1}  \frac{f(u)}{ \sqrt{\alpha -  u}}\; du = 0, \;\;
\forall \, \, 1 < \alpha \leq  \cosh^2 \rho_{\max}. $$
Since the left hand side as a function of $\alpha$ is smooth at $\cosh^2\rho_{\max}$, all its Taylor coefficients at this point must vanish. Thus
$$\int_{0}^{1}  f(u) \left (\cosh^2\rho_{\max} -  u \right )^{-n -\frac{1}{2}}\; du = 0 , \quad \forall n \in \mathbb N. $$ By the Stone-Weierstrass theorem, $f=0$, hence $P=0$. 
\end{proof}

\subsection{\label{FLAT} Infinitesimal rigidity and flatness}
In Section 3.2 of our earlier paper \cite{HeZ12}, we proved that infinitesimal rigidity
implies flatness, which completes the proof of Corollary
\ref{RIGIDCOR}:

\end{document}